\documentclass[reqno,11pt]{amsart}

\usepackage[margin=2.5cm]{geometry}

\usepackage{rotating}
\usepackage{amsmath}
\usepackage{amssymb}
\usepackage{amsthm}
\usepackage{graphicx}
\usepackage{tikz}
\usepackage{framed}
\usepackage{booktabs}
\usepackage{cleveref}
\usepackage{kbordermatrix}

\usepackage{algorithmic}

\newcommand{\symdiff}{\bigtriangleup}
\DeclareMathOperator{\ecycle}{ecycle}

\swapnumbers

\newtheorem{theorem}{Theorem}[section]
\newtheorem{lemma}[theorem]{Lemma}

\newcounter{claim_nb}[theorem]
\newtheorem{claim}[claim_nb]{Claim}

\newenvironment{cproof}
{\begin{proof}
 [Proof.]
 \vspace{-1.2\parsep}}
{ \end{proof}}

\newcommand\lref[1]{Lemma~\ref{lem:#1}}
\newcommand\tref[1]{Theorem~\ref{thm:#1}}
\newcommand\sref[1]{Section~\ref{sec:#1}}
\newcommand\clref[1]{Claim~\ref{cl:#1}}

\newcommand\ag{AG(3,2)}

\newcommand\ex[1]{\mathcal{EX}(#1)}

\newcommand\exa{\ensuremath{\ex{AG(3,2)}}}

\newcommand\mm{{\mathcal M}}
\newcommand\rk{\text{rank}}

\begin{document}

\title{Maximum size binary matroids with no $AG(3,2)$-minor are graphic}

\author{Joseph P. S. Kung}
\address{Department of Mathematics, University of North Texas, Denton, TX 76203, USA}
\email{kung@unt.edu}

\author{Dillon Mayhew}
\address{School of Mathematics, Statistics and Operations Research, Victoria University of Wellington, Wellington, New Zealand}
\email{dillon.mayhew@msor.vuw.ac.nz}

\author{Irene Pivotto}
\address{School of Mathematics and Statistics, University of Western Australia, Nedlands WA 6009, 
Australia}
\email{irene.pivotto@uwa.edu.au}

\author{Gordon F. Royle}
\address{School of Mathematics and Statistics, University of Western Australia, Nedlands WA 6009, 
Australia}
\email{gordon.royle@uwa.edu.au}





\begin{abstract}
We prove that the maximum size of a simple binary matroid of rank $r \geq 5$ with no $AG(3,2)$-minor is $\binom{r+1}{2}$ and characterise those matroids achieving this bound. When $r \geq 6$, the graphic matroid $M(K_{r+1})$ is the unique matroid meeting the bound, but there are a handful of smaller examples. In addition, we determine the
size function for non-regular simple binary matroids with no $AG(3,2)$-minor and characterise the matroids of maximum size for each rank.\end{abstract}

\maketitle

\section{Introduction}

In his survey paper ``Extremal Matroid Theory'', Kung \cite{MR1224696} describes the subject as being primarily concerned with instances of the following problem:

\medskip

\noindent
{\sc Fundamental Problem} 
{\em 
Let $\mathcal M$ be a class of matroids satisfying given properties. Determine the size function 
\[
h(\mm;r) = \max\{|E(M)| : M \in \mm \text{ and } \rk(M) = r\}
\]
and characterise the matroids of maximum size for each rank.}

\medskip
For example, for the class ${\mathcal R}$ of  {\em simple regular matroids}, an early theorem of Heller \cite{MR0094381} implies that
\[
h(\mathcal{R}; r) = \binom{r+1}{2},
\]
and from this, it quickly follows that the maximum size regular matroids are graphic. 

There are various classes of matroids, such as ``all simple $GF(q)$-representable matroids'', for which the fundamental problem is essentially trivial, but other than these cases there are relatively few classes of non-graphic matroids\footnote{These questions for subclasses of graphic matroids are the subject of extremal graph theory, and not further addressed here.} for which the fundamental problem is completely solved.  In view of the importance of {\em minors} in structural matroid theory, most of the exact results known are for various minor-closed classes of matroids, and mostly for matroids representable over a fixed finite field.  Here the Growth Rate Theorem of Geelen, Kung and Whittle \cite{MR2482959}  implies that any proper minor-closed class of simple $GF(q)$-representable matroids that does not include all $GF(q')$-representable matroids for some $q' \mid q$ has either  a {\em linear} or {\em quadratic} size function. The former occurs if any graphic matroid is excluded from the class, and the latter otherwise. 
However, the Growth Rate Theorem is a qualitative theorem and, while it gives the ``shape'' of the size function, it does not provide an exact solution. The most important fixed finite field is the {\em binary field}, but even for binary matroids only a few exact results are known, and these are mostly consequences of exact structural descriptions of the relevant classes of matroids. A selection of these results are given in Table~\ref{tab:binarysizes}, where $\mathcal{EX}(M_1, M_2, \ldots)$ denotes the class of simple binary matroids with no $M_1$-minor, $M_2$-minor etc.  Although a few more such exact characterisations are known, these are mostly for even smaller classes of matroids defined by excluding even more small matroids.

\begin{table}[t]
\renewcommand{\arraystretch}{1.3}
\begin{tabular}{clp{6cm}}
\toprule
Class & Size function $h(r)$ & Maximum-sized matroids \\
\midrule 
\midrule
Cographic &
$3(r-1)$ &
\parbox{6cm}{
Bond matroids of $3$-connected cubic multigraphs (see \cite{MR850567}).}\\
\midrule 
$\mathcal{EX}(F_7, F_7^*)$ &
$\binom{r+1}{2}$&
$M(K_{r+1})$ (\cite{MR0094381}).\\
\midrule
$\mathcal{EX}(W_4)$ &
$ \begin{cases}
3r-2, & r \text{ odd},\\
3r -3, & r \text{ even}.
\end{cases}$
 &\parbox{6cm}{Parallel connections of copies of $F_7$ and 
 the binary spike $Z_4$ found in Oxley~\cite{MR2849819} (\cite{MR879563}). }\\
\midrule
$\mathcal{EX}(K_{3,3})$&
$ \begin{cases}
14r/3 -7, & r = 0 \pmod{3},\\
14r/3-11/3, & r= 1 \pmod{3},\\
14r/3-19/3, & r=2 \pmod{3}.
\end{cases}$
&\parbox{6cm}{Parallel connections of copies of $PG(1,2)$, $PG(2,2)$ and $PG(3,2)$ (\cite{MR2742785}).}\\
\midrule
Conjecture for $\mathcal{EX}(K_5)$&
$ \begin{cases}
9r/2-13/2, & r \text{ odd},\\
9r/2 - 6, & r \text{ even}.
\end{cases}$
&\parbox{6cm}{Parallel connections along a line of copies of $C_{12}$ and $PG(2,2)$.}\\
\bottomrule
\end{tabular}
\caption{Some known and conjectured size functions for classes of binary matroids}
\label{tab:binarysizes}
\end{table}

In this paper we consider the class $\mathcal{EX}(AG(3,2))$, being the class of simple binary matroids with no $AG(3,2)$-minor and prove the following main result:

\begin{theorem}\label{thm:max_AG32-free}
Let $M$ be a simple binary matroid of rank $r \geq 5$ with no $\ag$-minor. Then $|M|\leq \binom{r+1}{2}$.
Moreover, if $r \geq 6$ and $|M|= \binom{r+1}{2}$ then $M$ is the graphic matroid of $K_{r+1}$.
\end{theorem}

Note that this theorem contradicts Theorem~6.8 of Kung \cite{MR1224696} where it is claimed that $\exa$  has a size function strictly larger than that of graphic matroids. 
This was an error caused by incorrectly extrapolating to higher ranks the construction of the matroid $M(K_5)^+$ which is the unique rank-4 extension of the graphic matroid $M(K_5)$, and the unique maximum-sized rank-4 matroid in \exa.

We were initially motivated to study the class $\exa$ due to its connections to Seymour's $1$-flowing conjecture, which asserts that a binary matroid is $1$-flowing  if and only if it has no $AG(3,2)$-, $T_{11}$- or $T_{11}^*$-minor (see Seymour \cite{MR633121} for details).  Regular matroids satisfy the $1$-flowing conjecture and so the non-regular members of this class are important for the resolution of this conjecture. Our second main result is the determination of the maximum sized {\em non-regular} matroids in $\mathcal{EX}(AG(3,2))$.

\begin{theorem}\label{thm:max_AG32-free_nonreg}
Let $M$ be a simple non-regular binary matroid of rank $r \geq 6$ with no $\ag$-minor. Then $|M|\leq \binom{r}{2}+4$.
Moreover, if $|M|= \binom{r}{2}+4$ and $r\geq 7$, then $M$ is the generalized parallel connection along a line of $M(K_r)$ and $F_7$.
\end{theorem}

It is quite common in such results that there are various maximum-sized ``sporadic'' matroids at lower ranks, and these can obscure the emergence of the asymptotic behaviour, and cause difficulties in framing an inductive proof. We deal with this ``low-level junk'' through computer analysis of the lower rank cases, a development that Kung recognised would be necessary when he observed that {\em ``because many small examples need to be examined, some computer assistance will be needed''.} We place particular emphasis on the 
correctness of these computations as they form the foundation of the proof, and ensure that they are performed at least twice independently.

The proof of \tref{max_AG32-free} is in \sref{maxproof}, while \tref{max_AG32-free_nonreg} in proved in \sref{non-reg}. \sref{lowrank} describes the computational results for the low rank cases that provide the base cases for the subsequent inductive proofs, while \sref{graft} presents fundamental results on even cycle matroids and graft matroids which are needed for the proofs of the main results.

\section{Low rank}\label{sec:lowrank}

To get a clean base for an inductive proof, a case-analysis of the low-rank situation is necessary, and this is best done with the aid of a computer. To increase confidence in this part of the proof, we ensure that all computations are done twice using independent software packages and approaches.

For binary matroids of any moderate fixed rank $r$,  it is possible to determine the cycle index of the group $PGL(r,2)$, and use P\'olya's Enumeration Theorem to calculate the number of orbits of this group on subsets of $PG(r-1,2)$ of each size, and thereby determine the exact numbers of binary matroids of rank up to $r$.  Although the total numbers of binary matroids grow extremely rapidly with rank, it is straightforward to produce a complete list of all 475499108 (including the empty matroid) binary matroids of rank up to 6. It is easy to verify that the matroids produced are pairwise non-isomorphic, and as the numbers of matroids of every rank and size agree with the theoretical numbers, we are very confident that this list is both correct and complete. Working with the much smaller subset consisting of the binary matroids of size up to 22, it is relatively straightforward to calculate the minors of each matroid that arise from a single deletion or contraction (and simplification). Then information about which matroids in this list have $AG(3,2)$ as a minor can be propagated throughout the catalogue, finally yielding a catalogue of all the matroids of rank 6 and size up to 22. 

A complementary approach is to construct a catalogue of {\em only} the matroids in ${\mathcal EX}(AG(3,2))$ by a 
`bootstrap' approach that grows the catalogue according to size. At each stage of the process, a binary matroid of size $k$ is extended to one of size $k+1$ by the addition of a point in the ambient vector space. Rather than testing directly whether the new matroid has an $AG(3,2)$-minor, each of its single element deletions and contractions are tested for membership in the catalogue constructed so far. By starting with the list of binary matroids of some fixed rank, and size at most 8, but excluding $AG(3,2)$ itself, this approach generates only those matroids with no $AG(3,2)$-minor. This approach was implemented twice, once in the computer algebra system GAP and once in C. The former uses only very well-tested standard functions for isomorph rejection and is very easy to verify, but manages only up to rank 7, while the latter is much faster and could manage rank 8, but for which external validation is harder.  

In any event, all approaches produced entirely consistent results for the low-rank case analysis, and these results are 
shown in Table~\ref{tab:lowrank}.

\begin{table}[t]
\begin{tabular}{ccccc}
\toprule
Rank&Max size&Number&Max size non-regular&Number\\
\midrule
3&7&1&7&1\\
4&11&1&11&1\\
5&15&4&15&3\\
6&21&1&19&13\\
7&28&1&25&1\\
8&36&1&32&1\\
\bottomrule
\end{tabular}
\caption{Numbers of maximum-sized matroids in $\exa$ for rank up to $8$}
\label{tab:lowrank}
\end{table}

This table shows that already at rank 6, the unique maximum sized matroid in $\exa$ is $M(K_6)$, and this provides the base case for the inductive proof
given in \Cref{sec:maxproof}.  For lower ranks, the extremal matroids are the Fano plane for rank three, $M(K_5)^+$ for rank four and at rank 5 there are three coextensions of $M(K_5)^+$ along with $M(K_6)$. The three non-regular examples are shown in Appendix~\ref{lowrank}.

For ranks 3, 4 and 5, the maximum sized matroids in $\exa$ already include non-regular ones as described previously. By rank 7, the parallel connection along a line of $M(K_6)$ and $F_7$ is the unique maximum sized non-regular matroid and this forms the base case for the inductive proof given in \sref{non-reg}. At rank 6, there are 12 additional examples, all of which are coextensions of the three non-regular maximum sized matroids of rank 5 and size 15.

\section{Even cycle matroids and graft matroids}\label{sec:graft}
In this section we introduce even cycle matroids and graft matroids. Even cycle matroids are used for the proofs of both the main theorems, while graft matroids are only used in the proof of \tref{max_AG32-free_nonreg}. The results about even cycle matroids are mostly folklore, but a complete presentation of these results can be found in Pivotto \cite{pivottothesis}.

A {\em signed graph} is a pair $(G,\Sigma)$, where $G$ is a graph and $\Sigma$ is a set of edges of $G$.
Let $A$ be the vertex-edge incidence matrix of $G$ and $S$ the row incidence vector of edges in $\Sigma$; let $A'$ be obtained from $A$ by adding row $S$.
Then the binary matroid represented by $A'$ is the {\em even cycle matroid} represented by $(G,\Sigma)$, denoted $\ecycle(G,\Sigma)$. 
Since the rows of $A$ span the cuts of $G$, replacing $\Sigma$ with $\Sigma \triangle C$, for a cut $C$ of $G$, does not change the matroid. We call any such set $\Sigma \symdiff C$ a {\em signature} of $(G,\Sigma)$ and the operation of replacing $\Sigma$ by $\Sigma\symdiff C$ is called a {\em resigning}.

Suppose that $M$ is a binary matroid and $M/x$ is graphic for some element $x$. Then any matrix representation of $M$ may be row-reduced to the following form

\begin{displaymath}
A= \kbordermatrix{ 
 & x &  &\\
 & 0 & \vrule & \\
 & \vdots & \vrule & \qquad & A' & \qquad &\\
 & 0 & \vrule & \\\cline{2-7}
 &1 &\vrule & \qquad & S & \qquad &},
\end{displaymath}

\bigskip
where $M(A')$ is a graphic matroid. Let $G$ be a graph representing $M(A')$ with the addition of a loop labelled $x$; let $\Sigma$ be the set of edges in $G$ which correspond to columns of $A$ which have a $1$ in the last row. Then $A$ represents $\ecycle(G,\Sigma)$, i.e. $M$ is the even cycle matroid of $(G,\Sigma)$. 

A set of edges $F$ in a signed graph $(G,\Sigma)$ is {\em $\Sigma$-even} if $|F \cap \Sigma|$ is even (and is {\em $\Sigma$-odd} otherwise). 
If every cycle in $G$ is $\Sigma$-even, then the empty set is a signature of $(G,\Sigma)$ and $\ecycle(G,\Sigma)$ is a graphic matroid. 
Let $X$ be the set of $\Sigma$-odd loops of $G$ and suppose that every $\Sigma$-odd cycle of $G$ not in $X$ uses a specific vertex $w$; then every non-loop cycle of $G$ that doesn't use $w$ is $\Sigma$-even, so for some signature $\Sigma'$, we have that $\Sigma'-X \subseteq \delta_G(w)$.
Construct a graph $H$ from $G$ by splitting the vertex $w$ into two vertices $w'$ and $w''$ so that the edges in $\delta(w)\cap \Sigma'$ are incident with $w'$, the edges in $\delta(w)- \Sigma'$ are incident with $w''$ and the edges in $X$ become edges between $w'$ and $w''$.
With this construction $\ecycle(G,\Sigma)$ is the graphic matroid of $H$. We will use this construction both in the proof of \tref{max_AG32-free} and \tref{max_AG32-free_nonreg}.

Minor operations on even cycle matroids correspond to minor operations on signed graphs in the following way. Let $(G,\Sigma)$ be a signed graph and $e$ an edge of $G$. Then we define $(G,\Sigma)\backslash e$ as $(G\backslash e,\Sigma-\{e\})$. Contraction is defined as follows: if $e$ is a loop of $G$ which is in $\Sigma$, then $(G,\Sigma)/e=(G\backslash e,\emptyset)$; otherwise we pick any signature $\Sigma'$ which does not contain $e$ and define $(G,\Sigma)/e=(G/e,\Sigma')$. The choice of $\Sigma'$ is not unique, but this is irrelevant for the matroid. With these definitions we have that $\ecycle(G,\Sigma)\backslash e=\ecycle((G,\Sigma)\backslash e)$ and $\ecycle(G,\Sigma)/ e= \ecycle((G,\Sigma)/e)$.

A {\em digon} is a pair of parallel edges of $(G,\Sigma)$ with different parities.

We will make use of the even cycle representation of $F_7$, which is given in Figure~\ref{fig:F7ecycle}.

\begin{figure}[htbp]
\begin{center}
\includegraphics[width=3cm]{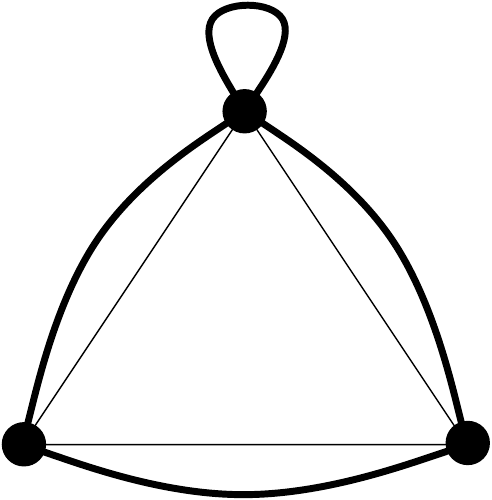}
\caption{Even cycle representation of $F_7$. Bold edges are odd, thin edges are even.}
\label{fig:F7ecycle}
\end{center}
\end{figure}

Next we introduce graft matroids, which will be used in the proof of \tref{max_AG32-free_nonreg}.
A {\em graft} is a pair $(G,T)$ where $G$ is a graph and $T$ a set of vertices of $G$ of even size.
The {\em graft matroid} arising from a graft $(G,T)$ is the binary matroid with matrix representation obtained from the vertex-edge incidence matrix of $G$ by adding a column that is the incidence vector of $T$.
The element corresponding to the incidence vector of $T$ is called the {\em graft element} of this matroid.
A {\em $T$-join} of a graft $(G,T)$ (or a $T$-join of $G$) is a set of edges $J$ of $G$ such that the vertices of odd degree in $G[J]$ are exactly the vertices in $T$. 
If $M$ is the graft matroid of $(G,T)$ and $t$ is the graft element of $M$, then the circuits of $M$ using $t$ are of the form $\{t\} \cup J$, for a minimal $T$-join $J$ of $G$ (i.e. a $T$-join of $G$ containing no cycle).

Since $F_7$ is the graft matroid of $(K_4,V(K_4))$, the maximum-sized matroids appearing in \tref{max_AG32-free_nonreg} are graft matroids for the following grafts $(G,T)$:
$G$ is the graph obtained from $K_n$ by adding a vertex $v$ adjacent to three vertices of $K_n$ and $T$ is the set comprised of $v$ and its neighbours. 

If $M$ is a graft matroid with graft element $t$ then $M\backslash t$ is obviously a graphic matroid, while $M/t$ may not even be a graft matroid. However, $M\backslash e$ and $M/e$ are graft matroids for every element $e\in E(M)-\{t\}$;
$M\backslash e$ is the graft matroid of $(G\backslash e,T)$ and $M /e$ is the graft matroid of $(G/e,T')$, where $T'$ is defined as follows: let $e=uv$ in $G$ and $w$ be the corresponding vertex in $G/e$; then $T'=T-\{u,v\}$ if both or neither of $u$ and $v$ are in $T$, and $T'=T-\{u,v\}\cup\{w\}$ otherwise. 

Consider the graft $(G,T)$, where $G$ is isomorphic to $K_{2,4}$ and $T$ is the set of vertices of degree two in $G$ (see Figure~\ref{fig:K24graft}). Then a $T$-join is given by $\delta(w)$, where $w$ has degree $4$. Let $M$ be the graft matroid of $(G,T)$; then $M$ has the following binary representation (where the first four basis elements are a star and $t$ is the graft element).

\begin{displaymath}
\kbordermatrix{
&  &  &  &  &  &  &  &  & t \cr
&1 & 0 & 0 & 0 & 0 & 1 & 1 & 1 & 1\cr
&0 & 1 & 0 & 0 & 0 & 1 & 0 & 0 & 1\cr
&0 & 0 & 1 & 0 & 0 & 0 & 1 & 0 & 1\cr
&0 & 0 & 0 & 1 & 0 & 0 & 0 & 1 & 1\cr
&0 & 0 & 0 & 0 & 1 & 1 & 1 & 1 & 0\cr}
\end{displaymath}

Pivoting on the first entry of the column for $t$ (i.e. adding the first row to the second, third and fourth) we obtain the following matrix.

\begin{displaymath}
\kbordermatrix{
&  &  &  &  &  &  &  &  & t \cr
&1 & 0 & 0 & 0 & 0 & 1 & 1 & 1 & 1\cr
&1 & 1 & 0 & 0 & 0 & 0 & 1 & 1 & 0\cr
&1 & 0 & 1 & 0 & 0 & 1 & 0 & 1 & 0\cr
&1 & 0 & 0 & 1 & 0 & 1 & 1 & 0 & 0\cr
&0 & 0 & 0 & 0 & 1 & 1 & 1 & 1 & 0\cr}
\end{displaymath}

From this representation it is clear that $M/t$ is isomorphic to $\ag$.

\begin{figure}[htbp]
\begin{center}
\includegraphics[width=4cm]{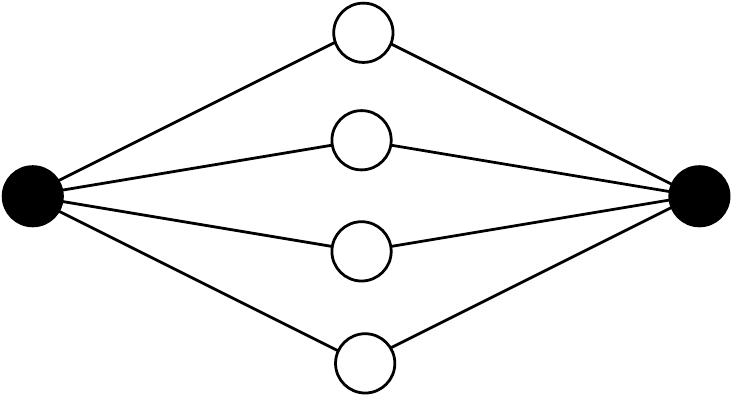}
\caption{The graft element is represented by white vertices. Contracting the graft element produces $\ag$.}
\label{fig:K24graft}
\end{center}
\end{figure}

\section{Proof of \tref{max_AG32-free}}\label{sec:maxproof}

By a {\em line} in a binary matroid $M$ we mean three elements forming a circuit of $M$.
Before proving \tref{max_AG32-free_nonreg} we require the following results.

\begin{lemma}\label{lem:lifting_circuit}
Let $x$ be an element in the simple binary matroid $M$.
Assume that $C$ is a circuit in a minor of $M/x$, where
$|C|=4$, and every element in $C$ is in a parallel pair in $M/x$.
Then $M$ has an $\ag$-minor.
\end{lemma}

\begin{proof}
Let $C=\{c_{1},c_{2},c_{3},c_{4}\}$, and let $\{c_{i},d_{i}\}$ be a parallel pair
of $M/x$ for every $i$, so that $\{x,c_{i},d_{i}\}$ is a line of $M$.
Let $N$ be a minor of $M/x$ such that $C$ is a circuit of $N$.
Consider a binary reduced representation of $N$ such that $c_{1}$, $c_{2}$, and $c_{3}$
all label rows, and extend this to a representation of $M/x$.
Now we can construct a representation of $M$ by adding a row labelled $x$.
Restricting this representation to $C\cup\{x,d_{1},d_{2},d_{3},d_{4}\}$ gives the
following submatrix.
\[
\kbordermatrix{&d_{1}&d_{2}&d_{3}&c_{4}&d_{4}\\
x&1&1&1&\alpha&\alpha+1\\
c_{1}&1&0&0&1&1\\
c_{2}&0&1&0&1&1\\
c_{3}&0&0&1&1&1
}
\]
We add the first row to all subsequent rows, and then delete the
column labelled by $c_{4}$ or $d_{4}$ that has $1$ as its first entry.
The resulting matrix represents $\ag$, so we are done.
\end{proof}

\begin{lemma}\label{lem:two_F7}
Let $M$ be a $3$-connected binary matroid that contains two distinct $F_7$-restrictions. Then $M$ contains an $\ag$-minor.
\end{lemma}

\begin{proof}
Assume that the lemma fails for $M$, where $M$ is as small as possible. 
Let $R_1$ and $R_2$ be the two $F_7$-restrictions. 
The hypotheses imply that $r(R_1 \cup R_2) > 3$, since M is binary and has no parallel pairs. 
If $r(R_1 \cup R_2)=4$, then submodularity implies that $cl(R_1) \cap cl(R_2)$ is a line. 
Now it is easy to verify that $(R_1 \cup R_2)\backslash (R_1 \cap R_2)$ is an $\ag$-restriction. 

Therefore $r(R_1 \cup R_2) > 4$. 
Because $M$ has no parallel pairs and $R_1$ and $R_2$ are projective geometries, $cl(R_1)=R_1$ and $cl(R_2)=R_2$.
If $R_1 \cup R_2 = E(M)$, then $(R_1, 
R_2\backslash R_1)$ is a $1$- or a $2$-separation of $M$. 
Therefore there is an element $e$ in neither $cl(R_1)$ nor $cl(R_2)$. 
Hence $R_1$ and $R_2$ are $F_7$-restrictions of $M\backslash e$ and $M/e$. 
Bixby's lemma says that either $si(M/e)$ or $co(M\backslash e)$ is $3$-connected. 
Assume $si(M/e)$ is $3$-connected.
When we suppress parallel pairs we delete elements other
than those in $R_1$ and $R_2$, or, if a parallel pair contains
$r_{1}\in R_{1}$ and $r_{2}\in R_{2}$,
we delete $r_{2}$ and replace $R_{2}$ with $(R_{2}-r_{2})\cup r_{1}$.
Note that $r_{M/e}(R_{1}\cup R_{2})>3$, so not every element
in $R_{1}$ is parallel with a point in $R_{2}$.
Therefore the resulting matroid is isomorphic to
$si(M/e)$, and has distinct $F_{7}$-restrictions,
$R_{1}$ and $R_{2}$.
By the minimality of $M$, it has an $\ag$-minor.
Therefore $co(M\backslash e)$ is $3$-connected.
If $R_1$ and $R_2$ are not $F_7$-restrictions of 
$co(M\backslash e)$, then there must be a series pair of
$M \backslash e$ contained in $R_1 \cup R_2$. 
But an $F_7$-restriction cannot intersect a series pair, by virtue of the fact that a circuit and a cocircuit cannot intersect in a single 
element.
Hence $co(M\backslash e)$ has an $\ag$-minor, and this
contradiction completes the proof.
\end{proof}

\begin{lemma}\label{lem:lines}
Let $M$ be a binary matroid and $\ell_1,\ldots,\ell_k$ be lines of $M$ all using an element $x$, where $k\geq 4$.
Suppose that there is no $F_7$-restriction of $M$ using $x$. If $r(\ell_1\cup \cdots\cup\ell_k)\leq k$, then $M$ contains an $\ag$-minor. 
\end{lemma}

\begin{proof}
Let $N$ be the restriction of $M$ to $\ell_1\cup\cdots\cup\ell_k$.
For $i=1,\ldots,k$, let $a_i$ and $b_i$ be the points in $\ell_i$ other than $x$.
Since $N/x$ has rank less than $k$, we may relabel the points in $\ell_1,\ldots, \ell_k$ so that there is a circuit $C$ contained in $\{a_1,\ldots, a_k\}$, say $C=\{a_1,a_2,\ldots,a_j\}$. 
If $j=3$, then $\ell_1\cup\ell_2\cup\ell_3$ is an $F_7$-restriction of $M$ using $x$. 
Therefore $j \geq 4$; then $\{a_{1},a_{2},a_{3},a_{4}\}$ is a circuit in $N/\{a_5,\ldots,a_j\}$, so
\lref{lifting_circuit} implies that $M$ has an $\ag$-minor.
\end{proof}

\begin{proof}[Proof of \tref{max_AG32-free}]
Assume by contradiction that the first part of the theorem fails. Let $M$ be a counterexample of minimum rank;
in other words, $M$ has no $\ag$-minor and $|M|>\binom{r+1}{2}$.
Computations show that the theorem holds for matroids of rank $5$, therefore $r\geq 6$ and for every minor $N$ of $M$ of rank $r-1$ we have $|N|\leq \binom{r}{2}$.

We first show that $M$ must be $3$-connected. 
If not, $M$ is the $1$- or $2$-sum of two matroids $N_1$ and $N_2$ of rank smaller than $r$. 
We consider only the case when $M$ is the $2$-sum of $N_1$ and $N_2$, as the other case is similar.
Let $r_1$ and $r_2$ denote the rank of $N_1$ and $N_2$ respectively. Then $r=r_1+r_2-1$.
$M$ has no parallel pairs, so we may assume $N_1$ has at most one parallel pair and $N_2$ is simple. 
By the minimality of $M$, $|N_1| \leq \binom{r_1+1}{2}+1$ and $|N_2| \leq \binom{r_2+1}{2}$, unless one or both of $r_1,r_2$ are equal to $3$ or $4$ (see Table~\ref{tab:lowrank}). We let the reader verify this case independently and just consider the case that $r_1,r_2 \notin \{3,4\}$.
Therefore 
\begin{align*}
|M|&=|N_1|+|N_2|-2 \leq \binom{r_1+1}{2} + \binom{r_2+1}{2} -1\\
&= \tfrac{1}{2}(r_1^2+r_1+r_2^2+r_2)-1\\
&=\binom{r+1}{2}+r_1+r_2-r_1r_2-1< \binom{r+1}{2}.
\end{align*}
This contradicts the assumption on the size of $M$. It follows that $M$ is $3$-connected.

Now consider any element $x$ of $M$ and let $N=si(M/x)$. 
Then the minimality of $M$ implies that $|N|\leq \binom{r}{2}$ and
\[|M|-|N| \geq \binom{r+1}{2}+1 - \binom{r}{2} = \tfrac{1}{2} (r^2+r-r^2+r) +1 = r+1.\]

Therefore in $M/x$ there are at least $r$ parallel pairs, so $x$ is in at least $r$ lines.  
Let $\ell_1,\ldots,\ell_r$ be such lines. The rank of $\ell_1\cup\cdots\cup\ell_r$ is at most the rank of $M$, which is $r$.
\lref{lines} implies that there is an $F_7$-restriction of $M$ using $x$.
The point $x$ was chosen arbitrarily, hence we showed that every point of $M$ is contained in an $F_7$-restriction. 
The rank of $M$ is at least $5$, so $M$ contains at least two distinct $F_7$-restrictions. It follows by \lref{two_F7} that $M$ has an $\ag$-minor, a contradiction. This concludes the proof of the first part of the theorem.

Now assume that the second part of the statement does not hold. Let $M$ be a counterexample of minimum rank, 
i.e. $|M|=\binom{r+1}{2}$ and $M$ is not graphic.
Computations show that the theorem holds for matroids of rank $6$, therefore $r\geq 7$ and the statement holds for every minor $N$ of $M$ of rank $r-1$.

By the same argument as the one above, $M$ is $3$-connected.
By \lref{two_F7}, there exists an element $x$ in $M$ which is not in an $F_7$-restriction. 
Let $N=si(M/x)$; then the size of $N$ is equal to $|M|-1$ minus the number of lines though $x$. 
Because $N$ is $\ag$-free and simple, $N$ has size at most $\binom{r}{2}$, so $x$ is in at least $r-1$ lines.
\lref{lines} implies that $x$ is in exactly $r-1$ lines $\ell_1,\ldots,\ell_{r-1}$ and $r(\ell_1\cup\cdots\cup\ell_{r-1})=r$.
In particular, $|N|=\binom{r}{2}$ and by the minimality of $M$, the matroid $N$ is the graphic matroid of $K_r$.

For $i=1,\ldots,r-1$, let $a_i$ be the point in $N$ that corresponds to the parallel pair $\ell_i - \{x\}$ in $M/x$. 
To simplify the exposition, we denote the graph $K_r$ representing $N$ by $G$ and we identify every edge of $G$ with the corresponding element in $N$. Let $A=\{a_1,\ldots,a_{r-1}\}$.
Since the rank of $\ell_1\cup\cdots\cup\ell_{r-1}$ is $r$, the set $A$ has rank $r-1$ in $N$. It follows that $A$ is a spanning tree in $G$.

Since $M$ is the lift of a graphic matroid, $M$ is an even cycle matroid; we describe its representation as a signed graph $(G',\Sigma)$ next.
Construct the graph $G'$ from $G$ as follows: add a loop (corresponding to the element $x$) and then, for every $a_i \in A$, add an edge parallel to $a_i$ in $G$; $\Sigma$ is comprised of $x$ plus all the elements of $M$ that are lifted by $x$ (i.e. the elements whose fundamental circuit for the basis $A \cup \{x\}$ contains $x$).
Then $M=\ecycle(G',\Sigma)$ and a line of $M$ is either a $\Sigma$-even triangle of $G'$ or the union of a $\Sigma$-odd loop and a digon.

We claim that $G[A]$ is a star of $G$; if not, it is possible to find an edge $e \in E(G)-A$ which is not in a triangle with two elements in $A$. Then the element $e$ is in at most $r-2$ $\Sigma$-even triangles in $G'$ and $e$ is not in any digon, so $e$ is in at most $r-2$ lines of $M$; this implies that $|si(M/e)|>\binom{r}{2}$, a contradiction. It follows that $G[A]$ is a star $\delta(w)$ of $G$. Next we show that every triangle in $G \backslash A$ is $\Sigma$-even. Suppose this is not the case, and choose an edge $e$ in such a triangle. Then $e$ is in at most $r-2$ $\Sigma$-even triangles in $(G',\Sigma)$ and $e$ is not in any digon; it follows that $e$ is in at most $r-2$ lines of $M$, again a contradiction. 
Therefore every cycle in $G\backslash A$ is $\Sigma$-even, since every cycle in $G\backslash A$ is the symmetric difference of triangles in $G \backslash A$. Since $(G'-w)\backslash x$ is equal to $G-A$, this implies that every odd cycle of $(G',\Sigma)$ (except for $x$) uses the vertex $w$.
By splitting $w$ as described in \sref{graft}, we construct a graphic representation of $M$, and we see that
$M$ is isomorphic to $M(K_{r+1})$.
\end{proof}

\section{Proof of \tref{max_AG32-free_nonreg}}\label{sec:non-reg}

Let $M$ be a simple non-regular binary matroid of rank $r$ with no $\ag$-minor. 
We need to show that $|M|\leq \binom{r}{2}+4$ (if $r\geq 6$) and, for $r\geq 7$, the only matroid achieving the bound is the generalized parallel connection along a line of $M(K_r)$ and $F_7$.

We prove the statement by induction on the rank. 
The base case (when the rank is $6$ or $7$) was proved computationally, as explained in \sref{lowrank}. 
Therefore we may assume that $r\geq 8$ and the statement holds for matroids of smaller rank than $M$.

The proof is split into two subsections; in the first one we show the bound on the size of $M$, while in the second, more technical one, we characterize the matroids achieving this bound.
To prove the bound on the size we proceed with a series of claims.
First we show that $M$ needs to be $3$-connected. Next we consider an element $x$ such that $M/x$ is non-regular; in \clref{manylines} we show that $x$ needs to be in many lines of $M$. The longest claim of this section, \clref{structure}, shows that if $x$ is in no $F_7$-restriction and $si(M/x)$ is the generalized parallel connection along a line of $M(K_{r-1})$ and $F_7$ then $M$ is what we would expect, namely the generalized parallel connection along a line of $M(K_r)$ and $F_7$.
\clref{rightsize} concludes the proof of the first half of the theorem, and \sref{size_nonreg}.

In \sref{char_nonreg} we characterize the non-regular matroids achieving the bound. 
We use the claims in \sref{size_nonreg} to show that there must exist an element $x$ in $M$ which is in $r-1$ lines that span $E(M)$.
Next we analyze the elements that are not in these lines. We show that $M$ must contain special elements that are not in the span of any two of the lines through $x$. We show that we must have exactly $r-3$ special elements and each one must be in exactly $r-2$ lines. 
Moreover we show that contracting any special element maintains non-regularity. 
Since each special element $y$ is in $r-2$ lines, the induction hypotheses imply that $si(M/y)$ is the generalized parallel connection of $M(K_{r-1})$ with $F_7$.
The fact that there are at least $r-3$ special elements, together with \lref{two_F7}, imply that at least one of them is not in an $F_7$-restriction, and we conclude the proof by \clref{structure}.

\subsection{Bound on the size}
\label{sec:size_nonreg}

In this section we prove the first part of \tref{max_AG32-free_nonreg}, i.e. we show that  $|M|\leq \binom{r}{2}+4$.

\begin{claim}\label{cl:3conn}
$M$ is $3$-connected.
\end{claim}

\begin{cproof}
If not, then $M$ is the $1$- or $2$-sum of two matroids $N_1$ and $N_2$ of rank smaller than $r$. 
We consider only the case when $M$ is the $2$-sum of $N_1$ and $N_2$, as the other case is similar.
Neither $N_{1}$ nor $N_{2}$ has an $\ag$-minor.
As $M$ is non-regular, at least one of $N_1$ and $N_2$ is non-regular. We let $N_1$ be non-regular.
Let $r_1$ and $r_2$ denote the rank of $N_1$ and $N_2$ respectively. Then $r=r_1+r_2-1$ and
$r_{1},r_{2}\geq 2$.
As $M$ has no parallel pairs, we may assume $N_1$ has at most one parallel pair and $N_2$ is simple. 
By the minimality of $M$, \tref{max_AG32-free} and Table~\ref{tab:lowrank}, $|N_1| \leq \binom{r_1}{2}+6$ and 
$|N_2| \leq \textrm{max}\{\binom{r_2+1}{2}, \binom{r_2}{2}+4\}$.
If $r_2=2$, $|N_2|$ is at most $3$ and for $r_2 \geq 3$, 
$\textrm{max}\{\binom{r_2+1}{2}, \binom{r_2}{2}+4\} \leq \binom{r_2+1}{2}+1$.
Therefore 
\begin{align*}
|M|&=|N_1|+|N_2|-2 \leq \binom{r_1}{2} + \binom{r_2+1}{2} +5\\
&= \tfrac{1}{2}(r_1^2-r_1+r_2^2+r_2)+5\\
&=\binom{r}{2}+r_1+2r_2-r_1r_2+4.
\end{align*}
As $|M|\geq \binom{r}{2}+4$, it follows that $r_{1}+2r_{2}-r_{1}r_{2}\geq 0$.
It is easily verified that this cannot occur, since $r_1+r_2 \geq 8$ and $r_1\geq 3$ (because $N_1$ is non-regular).
It follows that $M$ is $3$-connected.
\end{cproof}

\begin{claim}\label{cl:manylines}
If $|M|=\binom{r}{2}+4+p$ then each element $x$ of $M$ such that $M/x$ is non-regular is in at least $r-2+p$ lines.
\end{claim}

\begin{cproof}
Let $N=si(M/x)$. 
Then the induction hypotheses imply that $|N|\leq \binom{r-1}{2}+4$ and
\[|M|-|N| \geq \binom{r}{2}+4+p - \binom{r-1}{2}-4 = \tfrac{1}{2} (r^2-r-r^2+3r-2) +p = r-1+p.\]
Therefore $x$ is in at least $r-2+p$ lines.  
\end{cproof} 

\begin{claim}\label{cl:structure}
Suppose that $M$ contains an element $x$ such that $x$ is in no $F_7$-restriction and $si(M/x)$ is the parallel connection of $M(K_{r-1})$ and $F_7$. Then $M$ is the parallel connection of $M(K_r)$ and $F_7$.
\end{claim}

\begin{cproof}
Consider such an element $x$ and let $N=si(M/x)$.
\clref{manylines} implies that $x$ is in $k\geq r-2$ lines.
Let $\ell_1,\ldots,\ell_k$ be the lines through $x$.
 \lref{lines} and the fact that there is no $F_7$-restriction using $x$ implies that $r(\ell_1\cup \cdots \cup\ell_k)=k+1$.
Since $N$ is the generalized parallel connection of $M(K_{r-1})$ with $F_7$, $N$ is a graft matroid as described in \sref{graft}. More precisely, $N$ is the graft matroid of $(G,T)$, where $G$ is obtained from $K_{r-1}$ by adding a vertex $v$ adjacent to three vertices of the $K_{r-1}$ and $T$ is comprised of $v$ and its neighbours. Let the three neighbours of $v$ be $v_1,v_2,v_3$ and let $y_i=vv_i$ for $i=1,2,3$; denote by $z_1$ the edge forming a triangle with $y_2$ and $y_3$ and define $z_2$ and $z_3$ similarly. 
Let $Y=\{y_1,y_2,y_3\}$, $Z=\{z_1,z_2,z_3\}$ and $t$ be the graft element of $N$. Note that $Y\cup Z \cup \{t\}$ is an $F_7$-restriction of $N$ and the only lines containing $t$ in $N$ are of the form $\{t,y_i,z_i\}$, for $i=1,2,3$.

For $i=1,\ldots,k$, let $a_i$ be the point in $N$ that corresponds to the parallel pair $\ell_i - \{x\}$ in $M/x$. 
Let $A=\{a_1,\ldots,a_k\}$.
Since $r(\ell_1\cup \cdots \cup\ell_k)=k+1$ and $k\geq r-2$ we have that $A$ is an independent set of $N$ of size $r-1$ or $r-2$.
Since $Y \cup \{t\}$ is a circuit of $N$, there is an element in this set not contained in $A$.
It is easy to verify that $N\backslash y_{1}$ is graphic, and that in fact, $N\backslash y_{1}$
is equal to the matroid obtained from $M(G)$ by relabelling $y_{1}$ as $t$ and swapping the labels of $y_2$ and $y_3$.
Because $\{y_{1},y_{2},y_{3},t\}$ is a circuit of $N$, we can represent $N$ by
relabelling $G$ in this way, and then adding $y_{1}$ as a graft element
corresponding to the set $\{v,v_{1},v_{2},v_{3}\}$.
A similar symmetry applies to $y_{2}$ and $y_{3}$.
Because of these symmetries, we lose no generality in assuming that $t$ is not in $A$.

Now $M\backslash t$ is the lift of a graphic matroid (namely $M(G)$); that is,
$M\backslash t$ is an even cycle matroid.
Next we describe its representation as a signed graph $(G',\Sigma)$.
Construct the graph $G'$ from $G$ as follows: add a loop (corresponding to the element $x$) and then, for every $a_i \in A$, add an edge parallel to $a_i$ in $G$. Fix a basis $B$ of $M\backslash t$ containing $A$ and $x$; then $\Sigma$ is comprised of $x$ plus all the elements of $M$ whose fundamental circuit with $B$ contains $x$.
With this construction $M\backslash t=\ecycle(G',\Sigma)$, since $G'\backslash x$ is a graphic representation of $M\backslash t/x$ (see \sref{graft}).

To complete our proof we will prove that the following three
conditions hold:
\begin{itemize}
	\item[(a)] $G[A]$ is a star  $\delta_G(w)$;
	\item[(b)] every triangle of $G' - w$ is $\Sigma$-even;
	\item[(c)] $Y \cap A = \emptyset$.
\end{itemize}

First we show that these conditions imply the result.
Assume that (a), (b), and (c) hold.
Condition (b) implies that every cycle of $(G'-w)\backslash x$ is $\Sigma$-even.
Therefore there exists a signature $\Sigma'$ of $(G',\Sigma)$ such that $\Sigma' - \{x\}$ is contained in $\delta_{G'}(w)$.
For every $i \in [k]$, let $b_i$ be the edge of $G'$ parallel to $a_i$; since $\{a_i,b_i\}$ is a $\Sigma$-odd cycle of $G'$, it is also a $\Sigma'$-odd cycle. Therefore $\Sigma'$ contains $x$ and exactly one of $a_i,b_i$ for every $i$.
Let $H$ be the unsigned graph obtained from $G'$ by splitting $w$ into vertices
$r$ and $s$, where $x$ is the edge joining $r$ and $s$, and all edges
in $\Sigma' -\{x\}$ are incident with $r$, while each edge in $\delta_{G'}(w) - \Sigma'$ is incident with $s$.
Note that condition (c) implies $v$ is incident with exactly three edges in $H$, namely
$y_{1}$, $y_{2}$, and $y_{3}$ and $H - v$ is isomorphic to $K_r$.
Condition (b) implies that $M(H)=M\backslash t$, as discussed in \sref{graft}.
Thus $M$ is a graft of the form $(H,S)$, where $t$ is the graft element corresponding to $S$.
Since $M/x\backslash \{b_1,\ldots,b_k\}$ is the graft matroid of $(G,\{v,v_{1},v_{2},v_{3}\})$, we can assume that
either (i) $S=\{v,v_{1},v_{2},v_{3}\}$,
or (ii) $r$ or $s$ is in $\{v_{1},v_{2},v_{3}\}$, and $S=\{v,v_{1},v_{2},v_{3}\}\symdiff\{r,s\}$,
or (iii) neither $r$ nor $s$ is in $\{v_{1},v_{2},v_{3}\}$, and $S=\{v,v_{1},v_{2},v_{3},r,s\}$.
It is easily verified that in case (ii), $t$ is in exactly two lines in $M$, and that in
case (iii), $t$ is in no lines at all.
Therefore if (ii) or (iii) holds, then $|M/t|\geq |M|-3\geq\binom{r}{2}+1$, a contradiction to \tref{max_AG32-free}.
Hence case (i) holds, and $M$ is the generalized parallel connection of $M(K_r)$ with $F_7$.
We conclude that to complete the proof of the claim we only need to show that conditions (a), (b) and (c) hold.

We start by showing that $G[A-Y]$ is a star.
Since $A$ is an independent set in $N$, $G[A]$ is a forest.
Let $E'=E(G)-(Y\cup Z\cup A)$. Since $G'$ is obtained from $G$ by adding edges, we may see $E'$ as a subset of edges of $G'$.
For every edge $e \in E'$, there is an $F_{7}$-restriction,
$\{t\}\cup Y \cup Z$, in $N/e$, so $M/e$ is non-regular. It follows, by \clref{manylines}, that every $e \in E'$ is in at least $r-2$ lines in $M$.
No element in $E'$ is in a line with $t$ in $N$, since the only lines containing $t$ are given by $T$-joins of size two.
This implies there is no line of $M$ containing $t$ and an element of $E'$.
Moreover, since no elements of $A$ are in $E'$, no elements of $E'$ are in a digon in $G'$. 
Assume that $e \in E'$ is not in a triangle with two edges in $A$.
Observe that $e$ is not in a triangle of $G'$ with $v$, as $e\notin Y\cup Z$.
Since $G'-v$ has $r-1$ vertices, $e$ is in at most $r-3$ $\Sigma$-even triangles in $(G',\Sigma)$.
It follows that $e$ is in at most $r-3$ lines in $M$, a contradiction. 
Therefore every $e \in E'$ is in a triangle with two edges in $A$.
From this we deduce that $A-Y$ is a tree spanning $G-v$ and, since $G-v$ is isomorphic to $K_{r-1}$ and $r-1\geq 7$, there exists a leaf $u$ of $G[A-Y]$ not in $\{v,v_1,v_2,v_3\}$.
For every other leaf $u'$ of $G[A-Y]$, the edge $uu'$ is in $E'$, hence the $uu'$-path in $G[A-Y]$ must have length two.
It follows that $G[A-Y]$ is a star $\delta_G(w)$ spanning $G-v$.


Since every edge in $E'$ is in at least $r-2$ lines of $M$,
every triangle of $G'$ containing an edge of $E'$ and not using
the vertex $w$ is $\Sigma$-even.
This implies that if $w \notin \{v_1,v_2,v_3\}$ the triangle $\{z_1,z_2,z_3\}$ is also $\Sigma$-even, by the following argument.
Consider a vertex $u \in V(G')-\{v,v_1,v_2,v_3,w\}$ and let $H$ be the induced $K_4$ of $G'$ formed by $\{u,v_1,v_2,v_3\}$. Then every triangle of $H$ other than $Z$ is $\Sigma$-even. Since $Z$ is the symmetric difference of the other three triangles in $H$, $Z$ is also $\Sigma$-even.

Since $|A| \leq r-1$, the fact that $A-Y$ spans $G-v$ implies in particular that $|A \cap Y|\leq 1$.
If (c) fails we may assume that $y_1 \in A$.
If $y_{2}$ is in at most $2$ lines, then
$|M/y_{2}|\geq |M|-3\geq \binom{r}{2}+1$, and we have
a contradiction.
Therefore $y_2$, and by symmetry $y_3$, are in at least three lines of
$M$.
This means the triangles $\{z_1,y_2,y_3\}$,
$\{z_{2},y_{1},y_{3}\}$, and $\{z_{3},y_{1},y_{2}\}$
are $\Sigma$-even.
Therefore we may assume (up to resigning) that no edge of $E(G)-A$ is in $\Sigma$.

To conclude the proof of (c) (and consequently the proof of (a)), we need to consider two cases: either $w_1=v_1$ or not (recall that we are assuming that $y_1 \in A$ and $y_2,y_3 \notin A$). First we consider the case that $w\neq v_1$. By symmetry between $v_2$ and $v_3$, we may assume that $w \neq v_2$; let $(H,\Gamma)$ be the subgraph of $(G',\Sigma)$ induced by $\{v,v_1,v_2,w\}$ together with the loop $x$. Then removing parallel edges of the same parity from $(H,\Gamma) /y_2$ we obtain the signed graph representation of $F_7$ (see Figure~\ref{fig:F7ecycle}). Therefore $M/y_2$ is non-regular. It follows that $y_2$ must be in at least $r-2 \geq 6$ lines in $M$. However $y_2$ is in only two triangles in $G'$, so $y_2$ is in at most $4$ lines in $M$, a contradiction. It follows that this case does not occur, i.e. $w=v_1$. This (together with the previous paragraph) implies that every $\Sigma$-odd cycle of $G'$ (other than $x$) uses the vertex $v_1$. Hence $M\backslash t$ is the graphic matroid of the graph $H$ which is obtained from $G'$ by splitting $v_1$. It follows that $M$ is the graft matroid of $(H,S)$, for the appropriate choice of $S$. By the configuration of the edges in $A$, $H$ is obtained from $K_r$ by adding a vertex $v$ adjacent to four vertices of the $K_r$; the edges incident with $v$ are $y_1,y_2,y_3,x$.
Since $Y \cup \{t\}$ is a circuit of $N$, one of $Y\cup \{t\}$ or $Y \cup \{x,t\}$ is a circuit of $M$. First suppose that $Y \cup \{t\}$ is a circuit of $M$; in this case $Y$ is an $S$-join, i.e. $S$ is the set of vertices of odd degree in $H[Y]$. Therefore $(H,S)/x=(H',S')$, where $H'$ is isomorphic to $K_r$ and $|S'|=4$. Since $r\geq 6$, this graft contains the graft in Figure~\ref{fig:K24graft} as a minor; since the graft matroid of this graft contains $AG(3,2)$ as a minor (see \sref{graft}), this case cannot occur. The case where $Y \cup \{x,t\}$ is similar, only now $S$ is the set of neighbours of $v$ and  $(H',S')$ is obtained by deleting $v$.
This concludes the proof of conditions (a) and (c). 

To prove (b) it only remains to show that every triangle of $G' - w$ containing edges in $Y$ is $\Sigma$-even. 
If not, we may assume that $\{y_1,y_2,z_3\}$ is $\Sigma$-odd; since (c) holds, the only lines of $M$ through $y_1$ are $\{y_1,z_1,t\}$ and, possibly, $\{y_1,z_2,y_3\}$, so $y_1$ is in only two lines of $M$, a contradiction.
\end{cproof}

\begin{claim}\label{cl:rightsize}
$|M|=\binom{r}{2}+4$.
\end{claim}

\begin{cproof}
Suppose otherwise, i.e. $|M|\geq \binom{r}{2}+5$.
Let $x$ be any element of $M$ such that $M/x$ is non-regular; \clref{manylines} implies that $x$ is in at least $r-1$ lines.  
If $x$ is in $r$ (or more) lines, then \lref{lines} implies that there is an $F_7$-restriction of $M$ using $x$. 
It follows that for every element $x$ of $M$ such that $M/x$ is non-regular, either $x$ is in exactly $r-1$ lines or $x$ is in an $F_7$-restriction.
The matroid $M$ is non-regular, $3$-connected and is not isomorphic to $F_7^*$; the dual of Proposition 12.2.3 in~\cite{MR2849819} implies that $M$ has an $F_7$-minor. 
As $r\geq 8$, there is an independent set $I=\{x_1,x_2,x_3,x_4,x_5\}$ of $M$ such that $M/I$ has an $F_7$-minor (hence is non-regular). Each element in $I$ is either in an $F_7$-restriction or is in exactly $r-1$ lines. If every element in $I$ is in an $F_7$-restriction of $M$, then $M$ contains at least two distinct $F_7$-restrictions (because the rank of $F_7$ is three) and by \lref{two_F7} $\ag$ is a minor of $M$. It follows that there exists an element $x$ of $M$ such that $M/x$ is non-regular and $x$ is in exactly $r-1$ lines.
Therefore $|si(M/x)|\geq \binom{r}{2}+5-r=\binom{r-1}{2}+4$ and the inductive hypotheses imply that $si(M/x)$ is the generalized parallel connection of $M(K_{r-1})$ with $F_7$; \clref{structure} implies that $M$ is the generalized parallel connection of $M(K_r)$ with $F_7$.
However this is not possible, since we are assuming that $|M| \geq \binom{r}{2}+5$. We conclude that this case does not occur, i.e. $|M|=\binom{r}{2}+4$.
\end{cproof}

\subsection{The characterization}\label{sec:char_nonreg}
We conclude the proof of \tref{max_AG32-free_nonreg} by characterizing the non-regular matroids of maximum size. More specifically, we prove that $M$ is the generalized parallel connection along a line of $M(K_r)$ and $F_7$.

Consider an element $x$ of $M$ such that $M/x$ is non-regular. Let $N=si(M/x)$.
\clref{manylines} implies that $x$ is in at least $r-2$ lines.
We may again choose $x$ so that there is no $F_7$-restriction of $M$ using $x$.
Therefore (by \lref{lines}) $x$ is in either $r-2$ or $r-1$ lines. 
If $x$ is in exactly $r-2$ lines, then $|N|=\binom{r-1}{2}+4$, so $N$ is the generalized parallel connection of $M(K_{r-1})$ with $F_7$ and by \clref{structure} we conclude that $M$ is the generalized parallel connection of $M(K_r)$ with $F_7$.

It remains to consider the case when $x$ is in exactly $r-1$ lines $\ell_1,\ldots,\ell_{r-1}$, where $r(\ell_1\cup\cdots\cup\ell_{r-1})=r$.
Let $L=\ell_1\cup\cdots\cup\ell_{r-1}$. Then $M=cl(L)$.
For $i=1,\ldots,r-1$, let $a_i$ and $b_i$ be the points in $\ell_i$ other than $x$.
Suppose that every element of $M-L$ is on a line containing two points from $L$. Then in the standard binary matrix representation of $N$ with respect to the basis $\{a_1,\ldots,a_{r-1}\}$, every column has at most two non-zero entries; therefore $N$ is a graphic matroid. However, $x$ was chosen so that $N$ is non-regular. It follows that there exists an element $y$ such that $y$ is not on a line with two elements from $L$. We call such a $y$ a {\em special element} of $M$.
In the remainder of the proof we will analyse the special elements of $M$: we show that each special element $y$ is in exactly $r-2$ lines of $M$ and that $M/y$ is non-regular. Moreover, we will show that we have enough special elements to be able to choose one of them to not be in an $F_7$-restriction; this concludes the proof by induction and by \clref{structure}.

Let $y$ be any special element of $M$ and let $C$ be the fundamental circuit of $y$ with respect to the basis $\{x,a_1,\ldots,a_{r-1}\}$ of $M$. If $x$ and $a_i$ are in $C$, then we may consider, instead of $C$, the circuit $C \symdiff \ell_i$ and swap the labels of $a_i$ and $b_i$: by doing so we may assume that $x \notin C$.
Hence we may assume that $C=\{y,a_1,\ldots,a_k\}$, with $k \geq 3$. If $k\geq 4$ then in $M/\{y,a_5,\ldots,a_k\}$ the set $(\ell_1\cup \ell_2 \cup \ell_3\cup \ell_4)-\{x\}$ is an $\ag$-restriction. Therefore $k=3$ and $y$ is in a circuit with $a_1,a_2,a_3$. In particular, $M/y$ contains an $F_7$-restriction (with lines $\ell_1,\ell_2,\ell_3$).
It follows that for every special element $y$ of $M$ the following properties hold:

\begin{itemize}
	\item[\textbf{(S1)}] $y$ is in a circuit $C=\{y,c_i,c_j,c_k\}$ for distinct $i,j,k$, where $c_i \in \{a_i,b_i\}$,  $c_j \in \{a_j,b_j\}$ and  $c_k \in \{a_k,b_k\}$.
	\item[\textbf{(S2)}] $M/y$ is non-regular. 
\end{itemize}
Next we show that

\begin{itemize}
\item[\textbf{(S3)}] if $y$ is in a line $\{y,c,d\}$, then one of $c$ or $d$ is either in $L$ or is special.
\end{itemize}

Suppose that $y$ is in a line $\{y,c,d\}$, where $c,d\notin L$. To simplify the exposition, suppose that $\{y,a_1,a_2,a_3\}$ is a circuit as in (S1). If $c$ is not special, then $c$ is in some line $\{c,c_1,c_2\}$, where $c_1,c_2 \in L$ and $c_1,c_2$ are in distinct lines with $x$. 
It follows that $\{y,c,d\}\symdiff\{y,a_1,a_2,a_3\}\symdiff\{c,c_1,c_2\}=\{d\}\cup(\{a_1,a_2,a_3\}\symdiff\{c_1,c_2\})=D$ is a disjoint union of circuits of $M$; since the set $\{a_1,a_2,a_3\}\symdiff\{c_1,c_2\}$ does not contain any lines, $D$ is actually  a circuit of $M$. It remains to show that $\{c_1,c_2\} \not\subset \{a_1,a_2,a_3,b_1,b_2,b_3\}$, as then $d$ is special and (S3) holds.
Note that $\{y,a_1,a_2,a_3\} \symdiff \ell_1 \symdiff \ell_2=\{y,b_1,b_2,a_3\}$ is also a circuit of $M$ (and so are $\{y,a_1,b_2,b_3\}$ and $\{y,b_1,a_2,b_3\}$). Therefore, to show that $\{c_1,c_2\} \not\subset \{a_1,a_2,a_3,b_1,b_2,b_3\}$, we only need to consider three cases:
\begin{itemize}
	\item $\{c_1,c_2\}=\{a_1,a_2\}$: in this case $D=\{d,a_3\}$, but $M$ is a simple matroid.
	\item $\{c_1,c_2\}=\{a_1,b_1\}$: this implies that $c=x$, but $c$ is not in $L$.
	\item $\{c_1,c_2\}=\{a_1,b_2\}$: this would imply that $D=\{d,a_2,b_2,a_3\}$, hence $D\symdiff \ell_2= \{d,x,a_3\}$ is a line of $M$, but this is not possible, since $d \notin L$.
\end{itemize}
The next property we show for a special element $y$ is:

\begin{itemize}
\item[\textbf{(S4)}] $y$ is in at most two lines containing an element in $L$.
\end{itemize}
To prove (S4), we require the following technical result, which forbids certain configurations of pairs of special points.
We defer the proof of this result to the very end of the paper.

\begin{lemma}\label{lem:two_disj_special}
Let $M$ be a $3$-connected binary matroid. Suppose that an element $x$ of $M$ is in distinct lines $\ell_1,\ldots,\ell_{r-1}$ such that $r(\ell_1\cup\cdots\cup\ell_{r-1})=r$. 
Let $c_i$ be an element in $\ell_i -\{x\}$ for every $i$.
Suppose that there exist two elements $y_1$ and $y_2$ such that $\{y_1,c_1,c_2,c_3\}$ and one of
$\{y_2,c_3,c_4,c_5\}$ or $\{y_2,c_4,c_5,c_6\}$ are circuits of $M$.
Then $M$ contains an $\ag$-minor.
\end{lemma}

Now we can proceed to show that (S4) holds.
To simplify the exposition suppose that $C=\{c_1,c_2,c_3,y\}$ is a circuit as in (S1), where $c_i \in \ell_i$ for $i=1,2,3$.
First we show that $y$ is not in a line with $a_j$ or with $b_j$ for $j\geq 4$.
Suppose that $y$ is in a line $\ell=\{y,c_j,d\}$, where $c_j \in \ell_j$ and $j\geq 4$. 
The element $y$ is special, therefore $d \notin L$.
Therefore $\ell\symdiff C=\{d,c_1,c_2,c_3,c_j\}$ is a circuit and in $M/d$ the elements in $(\ell_1\cup\ell_2\cup\ell_3\cup\ell_j)-\{x\}$ form an $\ag$-restriction.
Next we show that $y$ is not in two lines $\{y,c,a_j\}$ and $\{y,d,b_j\}$; if this is the case, then $\{c,d,x\}$ is a line, therefore $c,d\in L$, contradicting the fact that $y$ is special.

It follows that if $y$ is in three lines with elements in $L$, then we may assume these lines are $\{y,a_1,d_1\}$, $\{y,a_2,d_2\}$, $\{y,a_3,d_3\}$, where $d_1,d_2,d_3 \notin L$. By possibly taking the symmetric difference of $C$ with two of the lines $\ell_1, \ell_2, \ell_3$, we may assume that $C=\{a_1,a_2,c_3,y\}$, where $c_3 \in \{a_3,b_3\}$.

If $c_3=b_3$, then consider the matrix representation for $M$ with respect to the basis $\{x,a_1,a_2,\ldots,a_{r-1}\}$. Contracting $a_4,\ldots,a_{r-1}$ we obtain the following submatrix:

 \begin{displaymath}
 \kbordermatrix{
 &x & a_1 & a_2 & a_3 & b_1 & b_2 & b_3 & y & d_1 & d_2 & d_3 \cr
 &1 & 0 & 0 & 0 & 1 & 1 & 1 & 1 & 1 & 1 & 1\cr
 &0 & 1 & 0 & 0 & 1 & 0 & 0 & 1 & 0 & 1 & 1\cr
 &0 & 0 & 1 & 0 & 0 & 1 & 0 & 1 & 1 & 0 & 1\cr
 &0 & 0 & 0 & 1 & 0 & 0 & 1 & 1 & 1 & 1 & 0\cr}
 \end{displaymath}
 
Adding the second, third and fourth rows to the first and deleting the columns corresponding to $a_1,a_2$ and $a_3$ we obtain the following matrix, which is a representation of $\ag$.
 
 \begin{displaymath}
 \kbordermatrix{
 &x &  b_1 & b_2 & b_3 & y & d_1 & d_2 & d_3 \cr
 &1 & 0 & 0 & 0 & 0 & 1 & 1 & 1\cr
 &0 & 1 & 0 & 0 & 1 & 0 & 1 & 1\cr
 &0 & 0 & 1 & 0 & 1 & 1 & 0 & 1\cr
 &0 & 0 & 0 & 1 & 1 & 1 & 1 & 0\cr}
 \end{displaymath}

 It follows that $c_3=a_3$. By (S2), $y$ is in at least $r-2$ lines. Moreover, at most three of these lines are with elements in $L$. By (S3) this implies that there exists at least another special element $y'$. Let $C'$ be a circuit for $y'$ as in (S1). By \lref{two_disj_special} (and possibly taking the symmetric difference of $C'$ with two of the lines spanning $y'$) we may assume that $a_1,a_2 \in C'$. If $b_3 \in C'$, then $C\symdiff C' \symdiff \ell_3=\{x,y,y'\}$ is a line of $M$, contradicting the fact that all the lines through $x$ are in $L$. Therefore we may assume that $C'=\{y',a_1,a_2,a_4\}$. It is routine to check that $\{x,y,a_3,b_1,b_2,b_4,d_1,d_2\}$ is an $\ag$-restriction of $M/y'$. We conclude that (S4) holds.

Let $S=\{y_1,\ldots,y_s\}$ be the set of all special elements of $M$.
We already showed that $S$ is non-empty.
Because of property (S2), $y_1$ is in at least $r-2$ lines. By (S3) and (S4), there are at least $r-4\geq 4$ special elements distinct from $y_1$. Therefore $s \geq 5$.
For every $t \in [s]$, let $C_t$ be a circuit containing $y_t$ as in (S1).
Note that if $C_t=\{y_t,c_i,c_j,c_k\}$, where $c_i \in \ell_i$ and so on, then $C_t \symdiff \ell_i \symdiff \ell_j$ is also a circuit for $y_i$ as in (S1); we will make repeated use of this simple observation.

The final property we require for special elements is the following.

\begin{itemize}
\item[\textbf{(S5)}] We may assume that $C_i=\{a_1,a_2,a_{i+2}\}$ for every $i=1,\ldots,s$. In particular $s\leq r-3$.
\end{itemize}
By property (S1), every special element $y$ is associated with a triple $\{i,j,k\}$, where $y$ is in the span of $\ell_i\cup \ell_j \cup \ell_k$.
For every special element $y_t$ let the associated triple be $T_t$.
First we show that for every distinct $i,j \in [s]$ the triples $T_i$ and $T_j$ share exactly two elements. 
By \lref{two_disj_special}, $|T_i \cap T_j|\geq 2$. If this quantity is $3$ then $\{x,y_1,y_2\}$ is a line of $M$, a contradiction. 
Therefore, for all distinct $i,j \in [s]$, $|T_i \cap T_j|= 2$. 
Since $s\geq 5$, this implies that there are indices $t_1,t_2 \in [r-1]$ such that $T_i \cap T_j=\{t_1,t_2\}$ for all distinct $i,j \in [s]$. By relabelling we may assume that $C_i=\{a_1,a_2,a_{i+2}\}$ for every $i=1,\ldots,s$, and property (S5) holds.

By (S2) and \clref{manylines}, every special element is in at least $r-2$ lines. By (S3) and (S4), every $y_i$ is in at most $(|S|-1)+2=|S|+1$ lines.
By (S5), $|S|\leq r-3$. Therefore $|S|=r-3$ and every special element is in exactly $r-2$ lines. 
Moreover $S$ is an independent set: if not, it contains a circuit $C$ and we may assume $C=\{y_1,\ldots,y_k\}$.
Now $C\symdiff C_1 \symdiff \cdots \symdiff C_k \subseteq \{a_1,\ldots,a_{k+2}\}$ is dependent, a contradiction.

Since $S$ is independent and $s\geq 4$, we may choose a special element $y$ that is not in an $F_7$-restriction, or we would obtain an $\ag$-minor by \lref{two_F7}.
Since $y$ is in exactly $r-2$ lines of $M$ and $M/y$ is non-regular (property (S2)),
we conclude the proof of \tref{max_AG32-free_nonreg} by \clref{structure}.


\bigskip
\begin{proof}[Proof of \lref{two_disj_special}]
Assume that the lemma fails for $M$, where $M$ is as small as possible. 
First we assume that $\{y_2,c_3,c_4,c_5\}$ is a circuit of $M$.
It follows easily from the hypotheses that $\{x,c_{1},c_{2},c_{3},c_{4},c_{5}\}$ is independent in
$M$, so $\{c_{1},c_{2},c_{3},c_{4},c_{5}\}$ is independent in $M/x$, and neither
$\{c_{1},c_{2},c_{3}\}$ nor $\{c_3,c_{4},c_{5}\}$ spans $x$ in $M$.
Hence $\{y_{1},c_{1},c_{2},c_{3}\}$ and $\{y_{2},c_3,c_{4},c_{5}\}$ are
circuits in $M/x$.
Their symmetric difference, $C=\{y_1,y_2, c_1, c_2, c_4, c_5\}$ is a disjoint union of circuits
in $M/x$.
If $C$ is not a circuit of $M/x$, there is a circuit in $C$ containing $y_{1}$, but not $y_{2}$.
The symmetric
difference of this circuit with $\{y_1,c_1,c_2,c_3\}$ is a circuit contained in $\{c_1,c_2,c_3,c_4,c_5\}$,
which is a contradiction as this set is independent in $M/x$.
Thus $C$ is a circuit of $M/x$, so $\{c_1,c_2,c_4,c_5\}$ is a circuit in $M/\{x,y_{1},y_{2}\}$.
Now \lref{lifting_circuit} implies the result.

Next we consider the case that $t=6$. 
Let $L_1=\ell_1\cup\ell_2\cup\ell_3\cup\{y_{1}\}$, $L_2=\ell_4\cup\ell_5\cup\ell_6\cup\{y_{2}\}$.
Assume there is an element $e\notin cl(L_{1}\cup L_{2})$.
Bixby's lemma says that $co(M\backslash e)$ or $si(M/e)$ is $3$-connected.
In the latter case, $(M/e)|(L_{1}\cup L_{2})=M|(L_{1}\cup L_{2})$, and there
are no parallel pairs of $M/e$ contained in $L_{1}\cup L_{2}$.
Thus $si(M/e)$ provides us with a smaller counterexample to the lemma, and we have
a contradiction.
Therefore $co(M\backslash e)$ is $3$-connected.
Assume there is a triad of $M$ that contains $e$ and an element from $L_{1}\cup L_{2}$.
Such a triad must intersect one of the lines $\ell_{1},\ldots, \ell_{6}$ or
one of $\{y_{1},c_{1},c_{2},c_{3}\}$ or $\{y_{2},c_{4},c_{5},c_{6}\}$ in a single
element, a contradiction.
Let $S$ be a set such that $co(M\backslash e)\cong M\backslash e/S$.
By the previous argument, $S\cap (L_{1}\cup L_{2})=\emptyset$.
If $X\subseteq L_{1}\cup L_{2}$, and $r(X\cup S)\ne r(X)+r(S)$, then there is a
circuit contained in $L_{1}\cup L_{2}\cup S$ that contains elements
from $S$, and any such circuit intersects a series pair of
$M\backslash e$ in a single element.
This shows that $(M\backslash e /S)|(L_{1}\cup L_{2})=M|(L_{1}\cup L_{2})$, so
$co(M\backslash e)$ is a smaller counterexample.
From this contradiction we conclude that $cl(L_{1}\cup L_{2})=E(M)$.

If $cl(L_1) \cup cl(L_{2}) = E(M)$, then  
$(cl(L_1),cl(L_{2})-cl(L_{1}))$ is a $2$-separation of $M$.
Therefore there is an element $e$ in neither $cl(L_1)$ nor $cl(L_{2})$.
Since $\{x,c_{1},\ldots,c_{6}\}$ is a basis of $M$, it follows that
$\{c_{1},\ldots, c_{6}\}$ is a basis of $M/x$, and $\{y_{1},c_{2},c_{2},c_{3}\}$
and $\{y_{2},c_{4},c_{5},c_{6}\}$ are circuits of $M/x$.
Let $C$ be the fundamental circuit of $e$ relative to
$\{c_{1},\ldots, c_{6}\}$ in $M/x$.
Then $C-e$ is not contained in $\{c_{1},c_{2},c_{3}\}$ or
in $\{c_{4},c_{5},c_{6}\}$, by our choice of $e$.
If $C'\subseteq C-e$ and $|C'|\geq 4$, then $C'$ is a circuit in
$M/((C\cup x)-C')$, and \lref{lifting_circuit} implies $M$ has an $\ag$-minor.
Therefore $|C-e|\leq 3$.
By relabelling as necessary, we can assume that either
\begin{itemize}
	\item[(a)] $C=\{e,c_1,c_4,c_5\}$, or
	\item[(b)] $C=\{e,c_1,c_4\}$.
\end{itemize}
In case (a) we let $C'$ be $C\symdiff\{y_{1},c_{1},c_{2},c_{3}\}=\{e,y_1,c_2,c_3,c_4,c_5\}$.
Otherwise we let $C'$ be  $C\symdiff\{y_{1},c_{1},c_{2},c_{3}\}\symdiff\{y_{2},c_{4},c_{5},c_{6}\}
=\{e,y_1,y_2,c_2,c_3,c_5,c_6\}$.
In either case, $C'$ is a disjoint union of circuits and $\{e,y_{1}\}\subseteq C'$.
Furthermore, $C'$ contains exactly two elements from each of $\{c_{1},c_{2},c_{3}\}$, and $\{c_{4},c_{5},c_{6}\}$.
We prove that $C'$ is a circuit of $M/x$.

Let $C''\subseteq C'$ be a circuit of $M/x$ that contains $y_{1}$.
Note that $C''\ne \{y_{1},c_{1},c_{2},c_{3}\}$, so $C''$ contains either
$e$ or $y_{2}$.
In case (a), $y_{2}\notin C'$ and $C''=C'$, for otherwise $C'-\{e,y_{1}\}= \{c_2,c_3,c_4,c_5\}$ is dependent.
Next consider case (b).
If $y_{2}\notin C''$, then $e\in C''$, and the circuit in $C'$ containing $y_{2}$ must be the unique circuit in $\{y_{2},c_{1},\ldots, c_{6}\}$,
namely $\{y_{2},c_{4},c_{5},c_{6}\}$.
This is impossible since $c_4 \notin C'$.
Therefore $y_{2}\in C''$.
If $e\notin C''$, then the circuit in $C'$ that contains $e$ must be $C$, which is impossible by the choice of $C'$.
Therefore $e,y_{2}\in C''$, so $C''=C'$ since $C'-\{e,y_{1},y_{2}\}$ is independent
in $M/x$.
Now $C'$ is a circuit in $M/x$, so $\{c_{1},\ldots, c_{6}\}$ contains a $4$-element circuit in either $M/\{x,e,y_{1}\}$ or $M/\{x,e,y_{1},y_{2}\}$.
The result follows by \lref{lifting_circuit}.  
\end{proof}

\appendix  

\section{Low rank examples}\label{lowrank}

At rank 6, the graphic matroid is the unique maximum sized matroid, but at rank 5 there
are 3 others; these three can all be described as follows: first take a basis in $GF(2)^4$ and add the third point
on every line determined by the six pairs of basis elements; clearly this is a representation of the matroid $M(K_5)$. 
Then add one further point, still in rank 4, yielding the matroid $M(K_5)^+$.

$$\left[ \begin{array}{ccccccccccccccc}
1&1&1&1&0&0&0&0&0&0&0\\
1&0&0&0&1&1&1&0&0&0&1\\
0&1&0&0&1&0&0&1&1&0&1\\
0&0&1&0&0&1&0&1&0&1&1\\
\end{array}\right]$$

Now embed $GF(2)^4$ as the hyperplane $x_1 = 0$ in $GF(2)^5$ and let $e = (1,0,0,0,0)^T$. Then construct three new matroids as follows: the first one is obtained by adding the points on the lines between $e$ and three non-collinear points in the Fano sub plane of $M(K_5)^+$. The second is obtained by adding the points on the lines between $e$ and any three original basis elements whose span is {\em not} the Fano sub plane of $M(K_5)^+$. 

The final one does not have $M(K_5)^+$ in the induced hyperplane, but one of the points is ``lifted'' into the fifth 
dimension, but rather there are {\em five} affine points, namely $e$ which is joined by long lines to three points in $M(K_5)^+$ and one additional point. The binary matrix shows it as well as a verbal explanation: the $4 \times 11$ portion of
each matrix is the original $M(K_5)^+$.

$$\left[ \begin{array}{ccccccccccc|cccc}
0&0&0&0&0&0&0&0&0&0&0&1&1&1&1\\
\hline
1&1&1&1&0&0&0&0&0&0&0&0&0&0&0\\
1&0&0&0&1&1&1&0&0&0&1&0&0&1&0\\
0&1&0&0&1&0&0&1&1&0&1&0&1&1&0\\
0&0&1&0&0&1&0&1&0&1&1&1&1&1&0\\
\end{array}\right]$$

\bigskip

$$\left[ \begin{array}{ccccccccccc|cccc}
0&0&0&0&0&0&0&0&0&0&0&1&1&1&1\\
\hline
1&1&1&1&0&0&0&0&0&0&0&0&0&1&0\\
1&0&0&0&1&1&1&0&0&0&1&0&0&0&0\\
0&1&0&0&1&0&0&1&1&0&1&0&1&0&0\\
0&0&1&0&0&1&0&1&0&1&1&1&0&0&0\\
\end{array}\right]$$

\bigskip

$$\left[ \begin{array}{ccccccccccc|cccc}
0&0&0&0&0&0&0&1&0&0&0&1&1&1&1\\
\hline
1&1&1&1&0&0&0&0&0&0&0&0&0&1&0\\
1&0&0&0&1&1&1&0&0&0&1&0&0&0&1\\
0&1&0&0&1&0&0&1&1&0&1&0&0&0&1\\
0&0&1&0&0&1&0&1&0&1&1&0&1&1&1\\
\end{array}\right]$$

\bibliographystyle{amsplain}
\bibliography{/Users/gordon/Dropbox/gordonmaster}

\providecommand{\bysame}{\leavevmode\hbox to3em{\hrulefill}\thinspace}
\providecommand{\MR}{\relax\ifhmode\unskip\space\fi MR }
\providecommand{\MRhref}[2]{%
  \href{http://www.ams.org/mathscinet-getitem?mr=#1}{#2}
}
\providecommand{\href}[2]{#2}
\begin{thebibliography}{1}

\bibitem{MR2482959}
Jim Geelen, Joseph P.~S. Kung, and Geoff Whittle, \emph{Growth rates of
  minor-closed classes of matroids}, J. Combin. Theory Ser. B \textbf{99}
  (2009), no.~2, 420--427. \MR{2482959 (2010f:05039)}

\bibitem{MR0094381}
I.~Heller, \emph{On linear systems with integral valued solutions}, Pacific J.
  Math. \textbf{7} (1957), 1351--1364. \MR{0094381 (20 \#899)}

\bibitem{MR850567}
Joseph P.~S. Kung, \emph{Numerically regular hereditary classes of
  combinatorial geometries}, Geom. Dedicata \textbf{21} (1986), no.~1, 85--105.
  \MR{850567 (87m:05056)}

\bibitem{MR1224696}
\bysame, \emph{Extremal matroid theory}, Graph structure theory ({S}eattle,
  {WA}, 1991), Contemp. Math., vol. 147, Amer. Math. Soc., Providence, RI,
  1993, pp.~21--61. \MR{MR1224696 (94i:05022)}

\bibitem{MR2742785}
Dillon Mayhew, Gordon Royle, and Geoff Whittle, \emph{The internally
  4-connected binary matroids with no {$M(K_{3,3})$}-minor}, Mem. Amer. Math.
  Soc. \textbf{208} (2010), no.~981, vi+95. \MR{2742785 (2011m:05083)}

\bibitem{MR2849819}
James Oxley, \emph{Matroid theory}, second ed., Oxford Graduate Texts in
  Mathematics, vol.~21, Oxford University Press, Oxford, 2011. \MR{2849819
  (2012k:05002)}

\bibitem{MR879563}
James~G. Oxley, \emph{The binary matroids with no {$4$}-wheel minor}, Trans.
  Amer. Math. Soc. \textbf{301} (1987), no.~1, 63--75. \MR{879563 (88b:05043)}

\bibitem{pivottothesis}
Irene Pivotto, \emph{Even cycle and even cut matroids}, Ph.D. thesis,
  University of Waterloo, 2011.

\bibitem{MR633121}
P.~D. Seymour, \emph{Matroids and multicommodity flows}, European J. Combin.
  \textbf{2} (1981), no.~3, 257--290. \MR{MR633121 (82m:05030)}

\end{thebibliography}
\end{document}